\documentclass[preprint,11pt]{elsarticle}
\usepackage{lineno,hyperref}
\modulolinenumbers[5]

\usepackage{color}
\usepackage{amsmath}
\usepackage{amsfonts,amsthm,amssymb}
\usepackage{amsfonts}
\usepackage{graphics}
\usepackage{pstricks}
\usepackage{graphicx}
\usepackage{cleveref}
\usepackage{extarrows,chngpage,array,float}
\usepackage{amssymb}
\usepackage{latexsym,bm}
\newtheorem{theorem}{Theorem}[section]
\newtheorem{lemma}[theorem]{Lemma}
\newtheorem{definition}[theorem]{Definition}
\newtheorem{proposition}[theorem]{Proposition}

\newtheorem{remark}[theorem]{Remark}

\allowdisplaybreaks   

\numberwithin{equation}{section}

\makeatletter
\def\ps@pprintTitle{}
\makeatother

\begin{document}

\makeatletter
\def\ps@pprintTitle{}
\makeatother

\begin{frontmatter}

\title{Irreducibility of interlace polynomials}

\author{Jungang Chen$^{a}$, \ \ Xian'an Jin$^{b,a}$ \ \ Tianlong Ma$^{c,}$\footnote{Corresponding author}
	\\[1ex]
\small  $^a$ School of Mathematical Sciences\\[-0.8ex]
\small Xiamen University\\[-0.8ex]
\small P. R. China\\
\small $^b$School of Mathematics and Statistics\\[-0.8ex]
\small Qinghai Minzu University\\[-0.8ex]
\small P. R. China\\
\small  $^c$ School of Science\\[-0.8ex]
\small Jimei University\\[-0.8ex]
\small P. R. China\\
\small\tt Email: jgchen@stu.xmu.edu.cn, xajin@xmu.edu.cn, tianlongma@aliyun.com}

\begin{abstract}
The factorisation of graph polynomials often reflects combinatorial decomposition. For a
nonempty loopless graph $G$, we first prove that the two-variable
interlace polynomial $q(G;x,y)$, introduced by Arratia, Bollob\'as and Sorkin, is irreducible over $\mathbb{C}[x,y]$ if and only if $G$ is
connected, exactly paralleling the classical irreducibility theorem for the Tutte polynomial.
The loopless hypothesis is essential: we construct an infinite family of connected looped
graphs whose two-variable interlace polynomials are reducible. For a nonempty graph
$G$, we prove that Courcelle's multivariate interlace polynomial
$C_G(u,v;\mathbf{x},\mathbf{y})$ is irreducible over
$\mathbb{C}[u,v,x_a,y_a:a\in V(G)]$ if and only if $G$ is connected.
\end{abstract}


\end{frontmatter}

\section{Introduction}

An important theme in the theory of graph and matroid polynomials is that algebraic
factorisation should reflect combinatorial decomposition. The classical example is the Tutte
polynomial. Let $T_M(x,y)$ denote the Tutte polynomial of a matroid $M$. If $M$ is the
direct sum of two matroids $M_1$ and $M_2$, then
\[
T_M(x,y)=T_{M_1}(x,y)T_{M_2}(x,y).
\]
Brylawski~\cite{Brylawski} conjectured in 1972 that the converse should also hold. This
conjecture was proved by Merino, de Mier and Noy~\cite{MerinoDeMierNoy}: if $M$ is
connected, then $T_M(x,y)$ is irreducible over $\mathbb{Z}[x,y]$; in fact, it is irreducible
over $\mathbb{C}[x,y]$.

This result suggests a natural question for other rank-nullity type graph polynomials.
Ellis-Monaghan, Goodall, Moffatt, Noble and
Vena~\cite{EllisMonaghanGoodallMoffattNobleVena} proved the analogue for the ribbon graph
polynomial $R(\mathbb{G};x,y)$: it is irreducible if and only if the embedded graph
$\mathbb{G}$ is neither a disjoint union nor a join. Irreducibility therefore detects a
factorisation finer than connectedness alone.
Goodall, Jouve and Sereni~\cite{GoodallJouveSereni} generalized the Merino--de~Mier--Noy
theorem to ranked sets via a functional form of the Brylawski relations. Their framework
unifies irreducibility results for several bivariate combinatorial polynomials, and
verifies the Bohn--Cameron--M\"uller conjecture (that the Galois group of $T(M;x,y)$ is
the full symmetric group) for infinite families of connected matroids.

In this paper we study analogous irreducibility questions for the two-variable interlace polynomial, introduced by 
Arratia, Bollob\'as and Sorkin. Since the interlace polynomial is defined through the adjacency matrix over
$\mathbb{F}_2$, multiple edges are indistinguishable from a single edge in its definition.
We therefore restrict attention throughout to graphs without multiple edges. 

The two-variable interlace polynomial
$q(G;x,y)$~\cite{ABS-two} (defined in Section~2) is multiplicative under disjoint unions,
so disconnected graphs yield reducible polynomials. The natural converse asks whether
reducibility can occur for connected graphs.


Our first result shows that, for loopless graphs, the answer is no.

\begin{theorem}\label{thm:two-var}
	Let $G$ be a loopless graph with at least one vertex. Then $q(G;x,y)$ is irreducible in $\mathbb{C}[x,y]$ if
	and only if $G$ is connected.
\end{theorem}

The proof of Theorem~\ref{thm:two-var} is quite different from the proof for the Tutte
polynomial. We use only two elementary properties of $q(G;x,y)$: the specialization
$q(G;t,t)=t^{|V(G)|}$, and the order of vanishing of $q(G;x,y)$ at $(0,0)$. These two facts
force any nontrivial factorisation to have too large an order of vanishing at the origin.

The loopless hypothesis in Theorem~\ref{thm:two-var} is essential. We give an infinite
family of connected looped graphs for which the two-variable interlace polynomial is reducible.

The correct irreducibility statement for looped graphs is obtained by keeping the vertex
variables separate. Courcelle's multivariate interlace polynomial
$C_G(u,v;\mathbf{x},\mathbf{y})$ (defined in Section~4) assigns a separate pair of
variables to each vertex, recording whether it is chosen and whether its loop is toggled.
Our second main result is as follows.

\begin{theorem}\label{thm:multi-var}
	Let $G$ be a graph with at least one vertex and loops allowed. Then $C_G(u,v;\mathbf{x},\mathbf{y})$ is irreducible in
	$\mathbb{C}[u,v,x_a,y_a:a\in V(G)]$ if and only if $G$ is connected.
\end{theorem}

Thus the multivariate interlace polynomial detects the correct decomposition of a looped graph: the
only possible factorisation comes from the connected components of $G$.

The paper is organized as follows. In Section~2 we prove
Theorem~\ref{thm:two-var}. In Section~3 we show that the loopless hypothesis cannot be
removed for Theorem~\ref{thm:two-var}. In Section~4 we prove Theorem \ref{thm:multi-var}. 


\section{The irreducibility of two-variable interlace polynomials}

All graphs considered in this paper are finite and have no multiple edges; loops are allowed unless otherwise stated.

For a graph $G=(V(G),E(G))$, the \emph{adjacency matrix} $\mathcal{A}(G)=(a_{ij})$ over $\mathbb{F}_2$ is the symmetric matrix whose diagonal entry $a_{ii}$ is $1$ if and only if vertex $i$ is looped, and whose off-diagonal entry $a_{ij}$ is $1$ if and only if $i$ and $j$ are joined by an edge. 

\begin{definition}\emph{\cite{ABS-two}}
	The two-variable interlace polynomial is
	\[
	q(G;x,y)=\sum_{S\subseteq V(G)}
	(x-1)^{r(\mathcal{A}(G[S]))}(y-1)^{n(\mathcal{A}(G[S]))},
	\]
	where $r(\mathcal{A}(G[S]))$ and $n(\mathcal{A}(G[S]))$ denote, respectively, the rank and nullity over
	$\mathbb{F}_2$ of the adjacency matrix of the induced subgraph $G[S]$. 
\end{definition}

For short, we write $r(\mathcal{A}(G[S]))$ and $n(\mathcal{A}(G[S]))$ as $r(G[S])$ and $n(G[S])$ respectively. 

It is
multiplicative under disjoint unions~\cite{ABS-two}:
\[
q(G_1\sqcup G_2;x,y)=q(G_1;x,y)q(G_2;x,y).
\]

For a nonzero polynomial $f(x,y)\in\mathbb{C}[x,y]$, write
\[
\operatorname{ord}_{(0,0)}f
=\min\{\,i+j : \text{the coefficient of }x^iy^j\text{ in }f\text{ is nonzero}\,\},
\]
the order of vanishing of $f$ at the origin. Equivalently, $\operatorname{ord}_{(0,0)}f$ is
the largest integer $k$ such that every monomial appearing in $f$ has total degree at
least~$k$.

For a graph $G$, let $c(G)$ denote the number of connected components of $G$.

The \emph{vertex-nullity interlace polynomial} \(q_N(G;y)\), introduced by Arratia, Bollob\'as, and Sorkin in~\cite{ABS-one}, 
is defined  by
\[
q_N(G;y)=\sum_{S\subseteq V(G)}(y-1)^{n(G[S])}.
\]
If \(G\) is loopless, then \(\mathcal{A}(G[S])\) is alternating over
\(\mathbb{F}_2\) for every \(S\subseteq V(G)\), and hence \(r(G[S])\) is even.
Therefore,
\[
q(G;0,y)
=\sum_{S\subseteq V(G)}
(-1)^{r(G[S])}(y-1)^{n(G[S])}
=q_N(G;y).
\]
Combining this identity with Remark~20 of~\cite{ABS-one}, the following lemma is clear.

\begin{lemma}\emph{\cite{ABS-one}}\label{lem:ord}
	If $G$ is a nonempty loopless graph, then
	\[
	\operatorname{ord}_{(0,0)} q(G;x,y)=c(G).
	\]
	In particular, if $G$ is
	connected, then $\operatorname{ord}_{(0,0)} q(G;x,y)=1$.
\end{lemma}

\begin{lemma}\label{lem:maximal-ideal}
	Let $K$ be a field and let $f\in K[x,y]$. Suppose that
	\[
	f(t,t)=t^N
	\]
	for some $N\geq 1$, that the highest homogeneous component of $f$ is a
	monomial, and that $\operatorname{ord}_{(0,0)}f=1$. Then $f$ is
	irreducible in $K[x,y]$.
\end{lemma}

\begin{proof}
	Suppose that $f=FH$ for nonconstant polynomials $F,H\in K[x,y]$, and let
	$F_d$ and $H_e$ be their highest homogeneous components. Then $d,e\geq 1$,
	and $F_dH_e$ is the highest homogeneous component of $f$.
	
	By assumption, $F_dH_e$ is a monomial, say $F_dH_e=cx^ry^s$ with
	$c\in K\setminus\{0\}$. Since $K[x,y]$ is a unique factorisation domain,
	every irreducible factor of $F_d$ and $H_e$ must be an irreducible factor
	of $cx^ry^s$. Up to multiplication by nonzero constants, the only such
	factors are $x$ and $y$. Hence
	\[
	F_d=\alpha x^ay^b
	\qquad\text{and}\qquad
	H_e=\beta x^{a'}y^{b'}
	\]
	for some nonzero $\alpha,\beta\in K$. Thus $F_d$ and $H_e$ are monomials.
	
	It follows that $F_d(t,t)$ and $H_e(t,t)$ are nonzero. Therefore
	$F(t,t)$ and $H(t,t)$ are nonzero polynomials of degrees $d$ and $e$,
	respectively. Since
	\[
	F(t,t)H(t,t)=t^N,
	\]
	unique factorisation in $K[t]$ gives
	$F(t,t)=\gamma t^i$ and $H(t,t)=\gamma^{-1}t^j$ for some
	$\gamma\in K\setminus\{0\}$ and $i,j\geq 1$. Hence
	$F(0,0)=H(0,0)=0$.
	
	Thus every monomial in both $F$ and $H$ has positive total degree, so every
	monomial in $FH$ has total degree at least $2$. Consequently,
	$\operatorname{ord}_{(0,0)}f\geq 2$, contradicting
	$\operatorname{ord}_{(0,0)}f=1$. Therefore $f$ is irreducible in $K[x,y]$.
\end{proof}

We now prove Theorem \ref{thm:two-var}.
\begin{proof}[Proof of Theorem~\ref{thm:two-var}]
	If $G$ is disconnected, write $G=G_1\sqcup G_2$ with $G_1$ and $G_2$ nonempty. By
	multiplicativity,
	\[
	q(G;x,y)=q(G_1;x,y)q(G_2;x,y),
	\]
	and both factors are nonconstant. Hence $q(G;x,y)$ is reducible.
	
	Conversely, assume that $G$ is connected. Let $n=|V(G)|$. Setting $x=y=t$ gives
	\[
	q(G;t,t)
	=\sum_{S\subseteq V(G)}(t-1)^{r(G[S])+n(G[S])}
	=\sum_{S\subseteq V(G)}(t-1)^{|S|}
	=t^{\,n}.
	\]
	The summand indexed by $S$ has total degree $|S|$ and highest homogeneous
	component $x^{r(G[S])}y^{n(G[S])}$. Consequently, the highest homogeneous
	component of $q(G;x,y)$ comes only from $S=V(G)$ and is the monomial
	$x^{r(G)}y^{n(G)}$.
	
	Since $G$ is nonempty, $n\geq1$, and by Lemma~\ref{lem:ord},
	$\operatorname{ord}_{(0,0)}q(G;x,y)=1$. Thus all the hypotheses of
	Lemma~\ref{lem:maximal-ideal} hold with $N=n$, and $q(G;x,y)$ is irreducible
	over $\mathbb{C}[x,y]$.
\end{proof}

\begin{remark}\label{rem:characteristic-zero}
	Lemma~\ref{lem:maximal-ideal} holds over any field. However, to apply
	Lemma~\ref{lem:ord}, we require $\operatorname{char}K=0$, since then the
	map $\mathbb{Z}\to K$ is injective and no nonzero coefficient of
	$q(G;x,y)$ vanishes. Hence $\operatorname{ord}_{(0,0)}q(G;x,y)$ is
	unchanged over $K$. In positive characteristic, coefficients of minimum
	total degree may vanish. Therefore, Theorem~\ref{thm:two-var} holds over
	every field of characteristic zero.
\end{remark}

\section{The loopless hypothesis is necessary}

A vertex with a loop attached is called a \emph{looped vertex}; otherwise it is
\emph{unlooped}. A \emph{looped graph} is a graph with at least one looped vertex. An \emph{ordinary neighbour} of a vertex $v$ is a vertex
$u\neq v$ such that $u$ and $v$ are adjacent. A looped vertex is not an ordinary neighbour of itself. 

Theorem~\ref{thm:two-var} shows that for loopless graphs, connectedness is equivalent to
irreducibility of $q(G;x,y)$. In this section we show that this fails for looped graphs: we
construct an infinite family of connected looped graphs whose two-variable interlace
polynomials are reducible.

To do this, we first establish the following lemma. 

\begin{lemma}\label{lem:false-twin}
	Let $G$ be a looped graph, and let $v\in V(G)$ be a looped vertex. Let $G'$ be obtained
	from $G$ by adding a new looped vertex $v'$ that is not adjacent to $v$ and has the same
	ordinary neighbours as $v$. Then
	\[
	q(G';0,y)=2q(G;0,y).
	\]
\end{lemma}

\begin{proof}
	Set $t=y-1$, and for a symmetric matrix $N$ over $\mathbb{F}_2$, define
	\begin{align*}
		w(N)=(-1)^{r(N)}t^{n(N)}.
	\end{align*}
	Then
	\begin{align*}
		&q(G';0,t+1)\\
		=&\sum_{S\subseteq V(G')}w(\mathcal{A}(G'[S]))\\
		=&\sum_{S\subseteq V(G)}(w(\mathcal{A}(G'[S]))+w(\mathcal{A}(G'[S\cup \{v'\}])))\\
		=&\sum_{S\subseteq V(G)\setminus \{v\}} \left(w(\mathcal{A}(G'[S]))+w(\mathcal{A}(G'[S\cup \{v\}]))+w(\mathcal{A}(G'[S\cup \{v'\}]))+w(\mathcal{A}(G'[S\cup \{v,v'\}]))\right).
	\end{align*}
	
	For any $S\subseteq V(G)\setminus \{v\}$, 
	let $B=\mathcal{A}(G[S])$, and let $c$ be the column vector
	recording the adjacency of $v$ to the vertices of $S$ in $G$. With $v$ listed first in $\mathcal{A}(G[S\cup\{v\}])$, we have
	\[
	\mathcal{A}(G'[S\cup\{v\}])=\mathcal{A}(G[S\cup\{v\}])=
	\begin{pmatrix}
		1 & c^T\\
		c & B
	\end{pmatrix}:=M_1.
	\]
	Since $v$ and $v'$ have the same neighbours in $G'$, we have 
	\[
	\mathcal{A}(G'[S\cup\{v'\}])=
	\begin{pmatrix}
		1 & c^T\\
		c & B
	\end{pmatrix}=M_1,
	\]
	where $v'$ is listed first in $\mathcal{A}(G'[S\cup\{v'\}])$. Clearly, 
	\[
	\mathcal{A}(G'[S\cup\{v,v'\}])=
	\begin{pmatrix}
		1 & 0 & c^T\\
		0 & 1 & c^T\\
		c & c & B
	\end{pmatrix}:=M_2,
	\]
	where $v$ and $v'$ are listed first in $\mathcal{A}(G'[S\cup\{v,v'\}])$.
	Note that 
	\[
	\begin{pmatrix}
		I_2 & \mathbf{0}\\
		D & I_s
	\end{pmatrix}
	M_2\begin{pmatrix}
		I_2 & D^T\\
		\mathbf{0} & I_s
	\end{pmatrix}=\begin{pmatrix}
		I_2 & \mathbf{0}\\
		\mathbf{0} & B+DD^T
	\end{pmatrix},
	\]
	where $I_l$ denotes the identity matrix of order $l$, $D=\begin{pmatrix}c&c\end{pmatrix}$ and $s=|S|$.  
	Since $DD^T=cc^T+cc^T=0$ over $\mathbb{F}_2$, we have
	$r(M_2)=2+r(B)$ and $n(M_2)=n(B)$. Therefore
	\[
	w(M_2)=(-1)^{r(M_2)}t^{n(M_2)}
	=(-1)^{2+r(B)}t^{n(B)}
	=w(B).
	\]
	Then  
	\begin{align*}
		q(G';0,t+1)=&2\sum_{S\subseteq V(G)\setminus \{v\}} \left(w(B)+w(M_1)\right)\\
		=&2\sum_{S\subseteq V(G)\setminus \{v\}} \left(w(\mathcal{A}(G[S]))+w(\mathcal{A}(G[S\cup \{v\}]))\right)\\
		=&2\sum_{S\subseteq V(G)} w(\mathcal{A}(G[S]))\\
		=&2q(G;0,t+1).
	\end{align*}
	This completes the proof. 
\end{proof}

Let \(a,b\geq 1\) be integers. Let $X=\{x_1,\ldots,x_a\}$, $Y=\{y_1,\ldots,y_b\}$ and $\{u,v,w\}$ be three disjoint sets. Define the graph \(G_{a,b}\) to be the looped graph with vertex set
\[
V(G_{a,b})=Y\sqcup X\sqcup \{u,v,w\},
\]
and every vertex in \(X\cup Y\) is looped. The set of ordinary edges of
\(G_{a,b}\) is
\[
\begin{aligned}
	\{x_i y_j : 1\leq i\leq a,\ 1\leq j\leq b\} \cup \{x_iw : 1\leq i\leq a\} \cup \{y_ju : 1\leq j\leq b\} \cup \{uv,vw\}.
\end{aligned}
\]

In particular, \(G_{1,1}\) is obtained from a \(5\)-cycle by adding a loop at each of two adjacent vertices. The following conclusion is clear by a direct computation from the definition.

\begin{proposition}\label{ex:base}
	\begin{align*}
		q(G_{1,1};x,y)=x(
		x^3y+2x^3-2x^2y-4x^2+3xy+x+y^2-y).
	\end{align*}
	Further, $q(G_{1,1};x,y)$ is reducible although $G_{1,1}$ is connected.
\end{proposition}

\begin{proposition}\label{prop:looped-family}
	For every pair of integers $a,b\geq 1$, the graph $G_{a,b}$ is connected and
	\[
	x\mid q(G_{a,b};x,y).
	\]
	In particular, $q(G_{a,b};x,y)$ is reducible in $\mathbb{Z}[x,y]$.
\end{proposition}

\begin{proof}
	It is clear that $G_{a,b}$ is connected.
	By Proposition~\ref{ex:base}, $q(G_{1,1};0,y)=0$. 
	Starting from \(G_{1,1}\), we obtain \(G_{a,b}\) by successively adding the vertices
	\[
	x_2,\ldots,x_a,y_2,\ldots,y_b.
	\]
	At the stage when \(x_i\) is added, for each \(2\leq i\leq a\), the vertex \(x_i\) is looped, is not adjacent to \(x_1\), and has the same ordinary neighbours as \(x_1\). Similarly, at the stage when \(y_j\) is added, for each \(2\leq j\leq b\), the vertex \(y_j\) is looped, is not adjacent to \(y_1\), and has the same ordinary neighbours as \(y_1\). Applying Lemma ~\ref{lem:false-twin} successively \(a+b-2\) times yields
	\[
	q(G_{a,b};0,y)
	= 2^{a+b-2}q(G_{1,1};0,y)
	= 0.
	\]
	Regard $q(G_{a,b};x,y)$ as a polynomial in $x$ with coefficients in $\mathbb{Z}[y]$. Write
	\[
	q(G_{a,b};x,y)=f_0(y)+xf_1(y)+x^2f_2(y)+\cdots+x^df_d(y),
	\]
	where $f_i(y)\in\mathbb{Z}[y]$. Since $f_0(y)=q(G_{a,b};0,y)=0$, we have 
	\[
	q(G_{a,b};x,y)
	=x\bigl(f_1(y)+xf_2(y)+\cdots+x^{d-1}f_d(y)\bigr).
	\]
	Thus $x\mid q(G_{a,b};x,y)$. The quotient is nonconstant, since
	$q(G_{a,b};t,t)=t^{|V(G_{a,b})|}$ with $|V(G_{a,b})|=a+b+3\geq 5$. Therefore
	$q(G_{a,b};x,y)$ is reducible in $\mathbb{Z}[x,y]$.
\end{proof}

Therefore, the family $\{G_{a,b}\}_{a,b\geq 1}$ shows that the loopless assumption in
Theorem~\ref{thm:two-var} is essential.

\section{The irreducibility of multivariate interlace polynomials}

For a graph $G$ and $B\subseteq V(G)$, let $G\nabla B$ be the looped graph obtained from $G$ by toggling
the loop status of every vertex in $B$. For disjoint subsets $A,B\subseteq V(G)$, define
\[
r_G(A,B)=r\bigl(\mathcal{A}\bigl((G\nabla B)[A\cup B]\bigr)\bigr)\]
and 
\[
n_G(A,B)=|A\cup B|-r_G(A,B).
\]
For short, we write $r_G(A,B)$ and $n_G(A,B)$ as $r(A,B)$ and $n(A,B)$ respectively, if the graph $G$ is clear. 

Let
\[
\mathbf{x}=\{x_a:a\in V(G)\},\qquad
\mathbf{y}=\{y_a:a\in V(G)\},
\]
and write
\[
x_A=\prod_{a\in A}x_a,\qquad
y_B=\prod_{b\in B}y_b.
\]

\begin{definition}\emph{\cite{Courcelle}}
	The full multivariate interlace polynomial of $G$ is
	\[
	C_G(u,v;\mathbf{x},\mathbf{y})
	=
	\sum_{\substack{A,B\subseteq V(G)\\ A\cap B=\varnothing}}
	x_Ay_B\,u^{r(A,B)}v^{n(A,B)}.
	\]
\end{definition}

Thus each vertex has three possible states: it is not chosen, it is chosen without toggling
its loop, or it is chosen after toggling its loop.

Clearly, for a graph $G$, the two-variable interlace polynomial $q(G;x,y)$ is obtained from $C_G(u,v;\mathbf{x},\mathbf{y})$ by the following substitution: $u=x-1$, $v=y-1$, $x_a=1$ and $y_a=0$ for all $a\in V(G)$.

In \cite{Courcelle}, the following result was given:
\[
C_G(u,v;\mathbf{x},\mathbf{y})
=
C_{G_1}(u,v;\mathbf{x},\mathbf{y})
C_{G_2}(u,v;\mathbf{x},\mathbf{y}),
\]
when $G$ is disjoint union of $G_1$ and $G_2$.

We shall use the following standard form of Gauss's lemma; see
\cite[Section~9.3]{DummitFoote}.

\begin{lemma}[Gauss's lemma]\label{lem:gauss}
	Let $R$ be a unique factorisation domain with fraction field $K$, and let
	$f\in R[z_1,\ldots,z_m]$ be primitive. Then $f$ is irreducible in
	$R[z_1,\ldots,z_m]$ if and only if it is irreducible in
	$K[z_1,\ldots,z_m]$.
\end{lemma}

We now prove Theorem~\ref{thm:multi-var}.

\begin{proof}[Proof of Theorem~\ref{thm:multi-var}]
	Suppose first that $G$ is disconnected. Then $G=G_1\sqcup G_2$ for two
	nonempty induced subgraphs $G_1$ and $G_2$. By multiplicativity,
	\[
	C_G(u,v;\mathbf{x},\mathbf{y})
	=
	C_{G_1}(u,v;\mathbf{x},\mathbf{y})
	C_{G_2}(u,v;\mathbf{x},\mathbf{y}).
	\]
	Both factors are nonconstant, since each $G_i$ has at least one vertex.
	Hence $C_G$ is reducible.
	
	Conversely, suppose that $G$ is connected. Let
	$R=\mathbb{C}[u,v]$ and $K=\mathbb{C}(u,v)$. We regard $C_G$ as a
	polynomial in the vertex variables $x_a,y_a$, with coefficients in $R$.
	Its constant coefficient is $1$, corresponding to $A=B=\varnothing$.
	Thus its coefficients have no common nonunit divisor in $R$, so $C_G$ is
	primitive. By Lemma~\ref{lem:gauss}, it is enough to prove that $C_G$ is
	irreducible in
	$
	K[x_a,y_a:a\in V(G)].
	$
	
	Assume, to the contrary, that $C_G=FH$ for nonconstant polynomials
	$F,H\in K[x_a,y_a:a\in V(G)]$. Since the constant coefficient of $C_G$
	is $1$, we have $F(0)H(0)=1$, where all vertex variables are set equal to
	zero. After multiplying one factor by a nonzero element of $K$ and the
	other by its inverse, we may assume that
	\[
	F(0)=H(0)=1.
	\]
	
	Let $\ell_a\in\{0,1\}$ denote the loop status of a vertex $a$. The
	monomial $x_a$ arises uniquely from $A=\{a\}$ and $B=\varnothing$, so its
	coefficient in $C_G$ is
	$u^{\ell_a}v^{1-\ell_a}$. Similarly, the monomial $y_a$ arises uniquely
	from $A=\varnothing$ and $B=\{a\}$, and its coefficient is
	$u^{1-\ell_a}v^{\ell_a}$. In particular, both coefficients are nonzero.
	
	The polynomial $C_G$ is multiaffine in the vertex variables: each
	$x_a$ and each $y_a$ has degree at most $1$. Since degrees in a fixed
	variable are additive under multiplication over an integral domain,
	\[
	\deg_z C_G=\deg_zF+\deg_zH
	\]
	for every vertex variable $z$. Thus no vertex variable can occur in both
	$F$ and $H$. Since every $x_a$ and $y_a$ occurs in $C_G$, each of them
	occurs in exactly one factor.
	
	We next show that $x_a$ and $y_a$ must occur in the same factor. Suppose,
	after interchanging $F$ and $H$ if necessary, that $x_a$ occurs in $F$
	and $y_a$ occurs in $H$. Then $F$ is independent of $y_a$ and $H$ is
	independent of $x_a$, and we may write
	$F=F_0+x_aF_1$
	and
	$H=H_0+y_aH_1$,
	where $F_1,H_1\neq 0$. The coefficient of $x_ay_a$ in $FH$, regarded as
	a polynomial in the remaining variables, is then $F_1H_1$, which is
	nonzero because the polynomial ring over $K$ is an integral domain.
	However, no monomial in the defining sum for $C_G$ is divisible by
	$x_ay_a$, since $A\cap B=\varnothing$. This is a contradiction.
	
	Consequently, the factorisation determines a partition
	$V(G)=P\sqcup Q$ such that the two variables $x_a,y_a$ occur in $F$ for
	$a\in P$, and in $H$ for $a\in Q$.
	
	We claim that no ordinary edge joins $P$ and $Q$. Suppose that $ab$ is
	such an edge, with $a\in P$ and $b\in Q$. Since $x_a$ occurs only in
	$F$ and the constant coefficient of $H$ is $1$, the coefficient of the
	monomial $x_a$ in $F$ must equal its coefficient in $C_G$, namely
	$u^{\ell_a}v^{1-\ell_a}$. Similarly, the coefficient of $x_b$ in $H$ is
	$u^{\ell_b}v^{1-\ell_b}$. The only way to obtain the exact monomial $x_ax_b$ in the product $FH$ is
	to multiply the $x_a$-term of $F$ by the $x_b$-term of $H$. Therefore,
	the coefficient of $x_ax_b$ in $FH$ is $u^{\ell_a+\ell_b}v^{2-\ell_a-\ell_b}$.
	
	On the other hand, the monomial $x_ax_b$ in $C_G$ arises uniquely from
	$A=\{a,b\}$ and $B=\varnothing$. Since $a$ and $b$ are adjacent, the
	corresponding adjacency matrix is
	\[
	\begin{pmatrix}
		\ell_a & 1\\
		1 & \ell_b
	\end{pmatrix}.
	\]
	Thus the two computations of the coefficient of $x_ax_b$ are
	\[
	\begin{array}{c|c|c}
		(\ell_a,\ell_b)
		& \text{coefficient in }FH
		& \text{coefficient in }C_G\\ \hline
		(0,0) & v^2 & u^2\\
		(1,0) & uv  & u^2\\
		(0,1) & uv  & u^2\\
		(1,1) & u^2 & uv
	\end{array}
	\]
	and they are different in every case. This contradicts $C_G=FH$.
	Therefore, no ordinary edge joins $P$ and $Q$.
	
	Since $G$ is connected, one of $P$ and $Q$ must be empty. If
	$P=\varnothing$, then $F$ contains no vertex variables and hence
	$F=F(0)=1$. Similarly, if $Q=\varnothing$, then $H=1$. Thus the
	factorisation is trivial. Therefore $C_G$ is irreducible over $K$. By Lemma~\ref{lem:gauss}, it is
	irreducible in
	\[
	R[x_a,y_a:a\in V(G)]
	=
	\mathbb{C}[u,v,x_a,y_a:a\in V(G)].
	\]
\end{proof}

\begin{remark}
	Theorem~\ref{thm:multi-var} should be compared with Theorem~\ref{thm:two-var}.
	For loopless graphs, the two-variable polynomial already detects connectedness. For looped
	graphs this is no longer true, as Section~3 shows. The full multivariate polynomial avoids
	this loss of information by keeping the two chosen states of each vertex separate.
	The proof also works with $\mathbb{C}$ replaced by an arbitrary field $K$:
	$K[u,v]$ is a unique factorisation domain, and the coefficient comparisons remain valid in $K(u,v)$.
\end{remark}

\section*{Acknowledgements}
This work is supported by the National Natural Science Foundation of China (Nos. 12571379, 12571366), Scientific Research Start-Up Foundation of Jimei University (No. ZQ2024116), Fujian Provincial Department of Education (No. JAT251074) and Natural Science
Foundation of Xiamen Municipality (No. 3502Z202673038).

\section*{Declarations}
\noindent
{ \bf Conflict of interest}
The authors declare that they have no conflict of interest.

\end{document}